\documentclass[12pt]{article}

\usepackage{amssymb}   
\usepackage{amsfonts}  
\usepackage{amsmath}   

\usepackage{tikz}
\usetikzlibrary{calc}
\usetikzlibrary{decorations.pathreplacing}
\usetikzlibrary{shapes,snakes}

\usepackage{setspace}
\usepackage{subfigure}
\usepackage{url}
\usepackage{graphicx}
\usepackage{booktabs}
\usepackage{ifthen}
\usepackage{enumerate}
\usepackage{comment}

\usepackage[margin=1in]{geometry}

\newtheorem{thm}{Theorem}[section]

\numberwithin{equation}{section}

\begin{document}

\title{On the Edge-Balanced Index Sets of Odd/Even Complete Bipartite Graphs} 

\author{
Hung Hua\thanks{Undergraduate Student, Department of Mathematics, Clayton State University, (\texttt{hhua@student.clayton.edu})}
\and
Christopher Raridan\thanks{Department of Mathematics, Clayton State University, (\texttt{ChristopherRaridan@clayton.edu})}}
   
\maketitle

\begin{abstract}
In 2009, Kong, Wang, and Lee began work on the problem of finding the edge-balanced index sets of complete bipartite graphs $K_{m,n}$ by solving the cases where $n=1$, $2$, $3$, $4$, and $5$, and also the case where $m=n$. In an article soon to be published, Krop, Minion, Patel, and Raridan concluded the edge-balanced index set problem for complete bipartite graphs with both parts of odd cardinality. In this paper, we conclude the problem for complete bipartite graphs where the larger part is of odd cardinality and the smaller is of even cardinality.
\\[\baselineskip] 
\\[\baselineskip]
\end{abstract}

\section{Introduction}

For a graph $G=(V,E)$ with vertex set $V$ and edge set $E$, a \textit{binary edge-labeling} is a surjection $f : E \to \{ 0,1 \}$. Let $i \in \{ 0,1 \}$. An edge labeled $i$ is called an \textit{$i$-edge} and let $e(i)$ denote the total number of $i$-edges in $G$ with respect to a binary edge-labeling $f$. In the case where $|e(1)-e(0)| \leq 1$, a binary edge-labeling is called \textit{edge-friendly}. Call the number of $i$-edges incident with a vertex $v$ the \textit{$i$-degree} of $v$, denoted $\deg_i(v)$, so that the degree of $v$ is $\deg(v) = \deg_1(v) + \deg_0(v)$. An edge-friendly labeling of $G$ will induce a (possibly partial) \textit{vertex-labeling} where a vertex $v$ is labeled $1$ when $\deg_1(v) > \deg_0(v)$, is labeled $0$ when $\deg_0(v) > \deg_1(v)$, and is unlabeled when $\deg_1(v) = \deg_0(v)$. Call a vertex labeled $i$ an \textit{$i$-vertex} and let $v(i)$ denote the total number of $i$-vertices in $G$ with respect to an edge-friendly labeling $f$. The \textit{edge-balanced index set} of $G$ is defined as
\begin{align*}
EBI(G) = \big\{ |v(1)-v(0)|: \text{over all edge-friendly labelings of $G$} \big\}.
\end{align*} 
More information about graph labelings can be found in Gallian's dynamic survey~\cite{GallianYYYY}.

The idea of a balanced labeling was introduced in 1992 by Lee, Liu, and Tan~\cite{LLT1992}. In 1995, Kong and Lee provided results concerning edge-balanced graphs~\cite{KL1995}. In~\cite{KWL2009}, Kong, Wang, and Lee introduced the problem of finding the $EBI$ of complete bipartite graphs by solving the cases where the smaller part has cardinality 1, 2, 3, 4, or 5, and the special case where both parts have the same cardinality, but left all other cases open. In~\cite{KMPR2014}, Krop, Minion, Patel, and Raridan concluded the edge-balanced index set problem for complete bipartite graphs with both parts of odd cardinality. 

A natural next step in the problem is to find the $EBI$ of the complete bipartite graphs where at least one part has even cardinality. In this paper, we conclude the problem for complete bipartite graphs where the larger part is of odd cardinality and the smaller is of even cardinality.

For positive integers $a,b$, where $a \leq b$, let $[a]$ denote the set of integers $\{ 1, \dots, a \}$ and let $[a,b]$ denote the set of integers $\geq a$ but $\leq b$; in the case where $a=1$, $[a]=\{1\}$, and in the case where $a=b$, $[a,a]=\{a\}$.  

Throughout the rest of this paper, let $K_{m,n}$ be a complete bipartite graph with part $A$ of cardinality $m$ and part $B$ of cardinality $n$, where $m$ is odd, $n$ is even, and $m>n \geq 2$. Let $q$ be the quotient when $m$ is divided by $\frac{n}{2}+1$ and let $r$ be the remainder. If $r=0$, then we further partition $A$ into sets $A_i$ and denote the vertices of $A_i$ by $v_j^i$, where $i \in [q]$ and $j \in \left[ \frac{n}{2}+1 \right]$. If $r \geq 1$, then we partition $A$ as in the $r=0$ case with the addition of another partition $A_*$ having vertices denoted $v_i^*$, where $i \in [r]$. Denote the vertices of $B$ by $u_i$, where $i \in [n]$. 

In the case where $n=2$, we have $q = \frac{m-1}{2} \geq 1$ and $r=1$. If $n \geq 4$ and $q=1$, then $m = \frac{n}{2}+1+r$, where $r \geq 2$; for, if $r=0,1$, then $m \leq \frac{n}{2}+2 \leq n$, a contradiction. In fact, if $n \geq 4$, $q=1$, and $r \geq 2$ is even, say $r=2j$ for some $j \geq 1$, then $\frac{n}{2}$ is even since $m$ is odd. That is, $n = 4k$ for some $k \geq 1$. We have $m>n$ if and only if $j \geq k$ and $r \leq \frac{n}{2}$ if and only if $j \leq k$; hence, $j=k$ and $ r = \frac{n}{2}$. Similarly, if $n \geq 4$, $q=1$, and $r \geq 3$ is odd, we have $r = \frac{n}{2}$. Thus, for $n \geq 4$, if $q=1$, then $m=n+1$ and $r = \frac{n}{2}$. 

Any edge-friendly labeling of $K_{m,n}$ has $e(1) = e(0) = \frac{mn}{2}$. Every vertex in $A$ has even degree so each vertex in this part is a $1$-vertex, a $0$-vertex, or is unlabeled, and every vertex in $B$ has odd degree so each is labeled either $1$ or $0$. Let $v_A(i)$ and $v_B(i)$ represent the number of $i$-vertices in $A$ and $B$, respectively, so that $v(i) = v_A(i)+v_B(i)$, where $i=0,1$. Without loss of generality, we assume that $v_A(1) \geq v_A(0)$ and $v_B(1) \geq v_B(0)$, which implies that $v(1) \geq v(0)$; with this assumption, every element in $EBI(K_{m,n})$ can be computed as $v(1)-v(0)$. Note that not every vertex in $B$ can be a $1$-vertex. If every vertex in $B$ were a $1$-vertex, then the number of $1$-edges incident with these vertices would be at least $n\left(\frac{m+1}{2}\right) > \frac{mn}{2}$, a contradiction. However, it is possible to have only one $0$-vertex in $B$ since $(n-1)\left(\frac{m+1}{2}\right) < \frac{mn}{2}$. We have similar results for the vertices in $A$.

\section{Two Particular Edge-Friendly Labelings}
\label{sec:2-e-f}

In this section, we describe two edge-friendly labelings $f,f'$ of $K_{m,n}$. The labeling $f$ will show that $0 \in EBI(K_{m,n})$ and $f'$ will show that $n-2 \in EBI(K_{m,n})$. When $n=2$, the labelings $f$ and $f'$ give the same index $0$, so we do not construct $f'$ when $n=2$. For larger values of $n$, we obtain two distinct edge-friendly labelings that give different indices. 

\subsection{Initializing the Labelings $f,f'$}
\label{subsec:initialize-f-f'}

Set $f(v_1^i u_j)=f'(v_1^i u_j) = 1$, where $i \in [q]$ and $j \in \left[ \frac{n}{2} \right]$, and label the remaining edges incident with each vertex $v_1^i$ by $0$. Then $v_1^i$ is unlabeled. If $r \geq 1$, set $f(v_1^* u_j)=f'(v_1^* u_j) = 1$, where $j \in \left[ \frac{n}{2} \right]$, and label the remaining edges incident with vertex $v_1^*$ by $0$, so that $v_1^*$ is unlabeled.

\subsection{The Labeling $f$}
\label{subsec:f}

After initializing $f$ as described above, we continue to label the edges of $K_{m,n}$ as follows to create an edge-friendly labeling $f$: For $i \in [q]$, $j \in \left[ 2, \frac{n}{2}+1 \right]$, and $k \in \left[ \frac{n}{2} \right]$, set $f(v_j^i u_k) = 0$ and label the remaining edges incident with each vertex $v_j^i$ by $1$, so that $v_j^i$ is unlabeled. If $r \geq 2$, for $j \in [2, r]$ and $k \in \left[ \frac{n}{2} \right]$, set $f(v_j^* u_k) = 0$ and label the remaining edges incident with vertex $v_j^*$ by $1$, so that $v_j^*$ is unlabeled. Note that under $f$, all edges in the graph have been labeled either $0$ or $1$ and $f$ is edge-friendly by construction. 

All vertices in $A$ are unlabeled. If $n=2$, then $\deg_1(u_1) = \deg_0(u_2) = q+1$ and $\deg_0(u_1) = \deg_1(u_2) = q$, so vertex $u_1$ is a $1$-vertex and $u_2$ is a $0$-vertex. For $n \geq 4$ and $i \in \left[ \frac{n}{2} \right]$, we have $\deg_0(u_i) = \deg_1(u_{i+\frac{n}{2}}) > \deg_1(u_i) = \deg_0(u_{i+\frac{n}{2}})$, so vertex $u_i$ is a $0$-vertex and $u_{i+\frac{n}{2}}$ is a $1$-vertex. Thus, $v(1)=v(0)=\frac{n}{2}$, which gives $0 \in EBI(K_{m,n})$. 

\subsection{The Labeling $f'$}
\label{subsec:f'}

By previous remarks, we know that it is possible to have an edge-friendly labeling of $K_{m,n}$ where one vertex of $B$, say $u_n$, is a $0$-vertex and the remaining vertices $u_1, \dots, u_{n-1}$ are $1$-vertices. 

After initializing $f'$ as described above, we continue to label the edges of $K_{m,n}$, where $n \geq 4$, as follows to create an edge-friendly labeling $f'$. We begin by labeling the edges (which are not already labeled) incident with vertex $u_n$ by $0$ so that $\deg_0(u_n)=m$ and $u_n$ is a $0$-vertex. The remaining edge labels will be determined based on the parity of $i \in [q]$. For $i \in [q]$ with $i$ odd, $j \in \left[2,\frac{n}{2}+1\right]$, $k \in \left[\frac{n}{2}-1\right]$, and $\ell=j+k-2$, set $f'(v_j^i u_\ell)=0$ and label the remaining unlabeled edges incident with each vertex $v_j^i$ by $1$, so that $v_j^i$ is unlabeled. For $i \in [q]$ with $i$ even, $j \in \left[2,\frac{n}{2}\right]$, $k \in \left[\frac{n}{2}\right]$, and $\ell = j+k-1$, set $f'(v_j^i u_\ell)=1$ and label the remaining unlabeled edges incident with each vertex $v_j^i$ by $0$, so that $v_j^i$ is unlabeled. For $i \in [q]$ with $i$ even, $j = \frac{n}{2}+1$, $k \in \left[\frac{n}{2}-1\right]$, and $\ell=j+k-1$, set $f'(v_j^i u_\ell)=1$ and $f'(v_j^i u_1)=1$, and label the remaining unlabeled edges incident with each vertex $v_j^i$ by $0$, so that $v_j^i$ is unlabeled. If $r \geq 2$, then similar to the even $i$ case, we set $f'(v_j^* u_\ell)=1$, where $j \in \left[2,r\right]$, $k \in \left[\frac{n}{2}\right]$, and $\ell = j+k-1$, and label the remaining unlabeled edges incident with each vertex $v_j^*$ by $0$, so that $v_j^*$ is unlabeled. Under $f'$, all edges in the graph have been labeled either $0$ or $1$, and since $e(0)=e(1)$, the constructed labeling $f'$ is edge-friendly. 

All vertices in $A$ are unlabeled, vertex $u_n$ is a $0$-vertex, and $\deg_1(u_i) > \deg_0(u_i)$ for $i \in [n-1]$, so vertex $u_i$ is a $1$-vertex. Thus, $v(1)=n-1$ and $v(0)=1$, which gives $n-2 \in EBI(K_{m,n})$. 

\section{Main Result}
\label{sec:main-result}

We are now ready to prove the following:
\begin{thm}
\label{thm:max-EBI-odd-even}
Let $K_{m,n}$ be a complete bipartite graph with parts of cardinality $m$ and $n$, where $m$ is odd, $n$ is even, and $m>n \geq 2$. Then $EBI(K_{m,2}) = \{0\}$. For $n\geq 4$, let $q$ be the quotient when $m$ is divided by $\frac{n}{2}+1$ and let $r$ be the remainder. Then
\begin{align}
EBI(K_{m,n}) =
\begin{cases}
\left\{ 0,1, \dots, m+n-2q-2 \right\}, &\text{if~$r = 0$}, \\
\left\{ 0,1, \dots, m+n-2q-3 \right\}, &\text{if~$r = 1$}, \\
\left\{ 0,1, \dots, m+n-2q-4 \right\}, &\text{if~$r \geq 2$}.
\end{cases}
\end{align}
\end{thm}

\begin{proof}
Let $n=2$. Then the labeling $f$ given in Section~\ref{subsec:f} shows that $0 \in EBI(K_{m,2})$. To see that $0$ is the only index in the edge-balanced index set of $K_{m,2}$, consider switches on pairs of edges incident with a vertex $u \in B$, say $e=uv$ and $e'=uv'$, where $f(e)=1$ and $f(e')=0$. Such switches will not alter the edge-friendliness of the labeling, nor alter the label on vertex $u$, but each switch will change the unlabeled vertex $v$ to a $0$-vertex and the unlabeled vertex $v'$ to a $1$-vertex. No matter how many switches are performed (up to labeling all but one vertex in part $A$), we will always have $v_A(1)=v_A(0)$ and $v_B(1)=v_B(0)=1$. It is impossible to label all the vertices in part $A$, so $EBI(K_{m,2}) = \{ 0 \}$.


For the remainder of the proof, let $n \geq 4$, let $q$ be the quotient when $m$ is divided by $\frac{n}{2}+1$, and let $r$ be the remainder.


Consider the labeling $f'$ given in Section~\ref{subsec:f'}, which provides an edge-friendly labeling of $K_{m,n}$ and shows that $n-2 \in EBI(K_{m,n})$. We perform edge label switches on pairs of $0$-edges and $1$-edges incident with the same vertex in part $B$, noting that such a switch will not alter the edge-friendliness of the labeling. For $i \in [q]$, switch the label on edge $v_1^i u_1$ with the label on edge $u_1 v_2^i$. These edge label switches will not change the label on vertex $u_1$, but will cause $v_1^i$ to change from an unlabeled vertex to a $0$-vertex and will cause $v_2^i$ to change from unlabeled vertex to a $1$-vertex. After performing these edge label switches, we note that the number of $1$-vertices increased by $q$ and the number of $0$-vertices increased by $q$, so we still have $n-2 \in EBI(K_{m,n})$. Continuing our $\{0,1\}$-edge-pair switches, for $i \in [q]$ and $j \in \left[2,\frac{n}{2}\right]$, switch the label on edge $v_1^i u_j$ with the label on edge $u_j v_{j+1}^i$. Each such switch increases the edge-balanced index by one. Moreover, after all of these $\{0,1\}$-edge-pair switches, we have that $\deg_0(v_1^i)=n$, implying that each $v_1^i$ is a $0$-vertex with no incident $1$-edges, and that $\deg_1(v_j^i) = \frac{n}{2}+1$, where $j \in \left[2,\frac{n}{2}+1\right]$, implying that each $v_j^i$ is a $1$-vertex (but just barely). Thus, all vertices in $A_i$, where $i \in [q]$, are labeled either $0$ or $1$, and we have attained each index from $n-2$ to $n-2+q\left(\frac{n}{2}-1\right)=m+n-2q-2-r$ in the edge-balanced index set. If $r=0$, then we are done as we cannot increase the edge-balanced index further. That is, if $r=0$, then the maximal index in the edge-balanced index set is $m+n-2q-2$. Similarly, if $r=1$, then we do not have any extra $1$-edges incident with vertices $v_j^i$, where $i \in [q]$ and $j \in \left[2,\frac{n}{2}+1\right]$, that could be used to change vertex $v_1^*$ into a $1$-vertex, so we cannot increase the edge-balanced index further. That is, for $r=1$, the maximal index is $m+n-2q-3$. For values of $r \geq 2$, we may perform additional $\{0,1\}$-edge-pair switches to force all vertices in part $A$ to be labeled, increasing the number of $0$-vertices in $A$ by one and the number of $1$-vertices by $r-1$. In particular, for $j \in [r-1]$, switch the label on edge $v_1^* u_j$ with the label on edge $u_j v_{j+1}^*$. Then $v_1^*$ is a $0$-vertex and $v_{j+1}^*$ is a $1$-vertex. In this case, we have that $v_A(0)=q+1$ and $v_A(1)=m-q-1$, which means that the maximal index in the edge-balanced index set is $v(1)-v(0)=v_A(1)+v_B(1)-v_A(0)-v_B(0)=m-q-1+n-1-(q+1)-1=m+n-2q-4$. 


Now, consider the labeling $f$ given in Section~\ref{subsec:f}. Performing the same $\{0,1\}$-edge-pair switches described above, we find that we are able to achieve subsets of the edge-balanced index set based on the value of $r$. If $r=0$, then we achieve the indices $\{ 0,1, \dots, m-2q \}$. If $r=1$, then we achieve the indices $\{ 0,1, \dots, m-2q-1 \}$. If $r \geq 2$, then we achieve the indices $\{ 0,1, \dots, m-2q-2 \}$.  


For the last part of the proof, note that if $r=0$, then we have that $\{ 0,1, \dots, m-2q \} \cup \{ n-2, \dots, m+n-2q-2 \} = \{ 0,1, \dots, m+n-2q-2 \}$, since $q = \frac{2m}{n+2}$ and $m > n+1$ implies $m-2q = \frac{m(n-2)}{n+2}> n-3 + \frac{4}{n+2} > n-3$. That is, if $r=0$, then $EBI(K_{m,n}) = \{ 0,1, \dots, m+n-2q-2 \}$. Now, if $r=1$, then $q = \frac{2m-2}{n+2}$ and $m \geq n+3$ implies $m-2q-1 = \frac{(m-1)(n-2)}{n+2} \geq n-2$, and $EBI(K_{m,n}) = \{ 0,1, \dots, m+n-2q-3 \}$. Finally, for values of $r \geq 2$, we consider three cases: (i)~$m=n+1$, (ii)~$m=n+3$, and (iii)~$m \geq n+5$. For case~(i), if $m=n+1$, then $q=1$ and $m-2q-2=n-3$. For case~(ii), if $m=n+3$, then $q=2$ and $m-2q-2=n-3$. For case~(iii), if $m \geq n+5$, then $m-2q-2 = \frac{(m-2)(n-2)}{n+2} + \frac{4r-8}{n+2} \geq \frac{(n+3)(n-2)}{n+2} > n-2$. Thus, if $r \geq 2$, then $m-2q-2 \geq n-3$ and $EBI(K_{m,n}) = \{ 0,1, \dots, m+n-2q-4 \}$.
\end{proof}


\begin{thebibliography}{1}

\bibitem{GallianYYYY}
Joseph~A. Gallian.
\newblock A dynamic survey of graph labeling.
\newblock {\em Electron. J. Combin.}, 5:DS 6,
  \texttt{http://www.combinatorics.org/ojs/index.php/eljc/index}, 2013.

\bibitem{KL1995}
M.~C. Kong and Sin-Min Lee.
\newblock On edge-balanced graphs.
\newblock In {\em Graph {T}heory, {C}ombinatorics, and {A}lgorithms, {V}ol.\ 1,
  2 ({K}alamazoo, {MI}, 1992)}, Wiley-Intersci. Publ., pages 711--722. Wiley,
  New York, 1995.

\bibitem{KWL2009}
Man~C. Kong, Yung-Chin Wang, and Sin-Min Lee.
\newblock On edge-balanced index sets of some complete {$k$}-partite graphs.
\newblock In {\em Proceedings of the {F}ortieth {S}outheastern {I}nternational
  {C}onference on {C}ombinatorics, {G}raph {T}heory and {C}omputing}, volume
  196, pages 71--94, 2009.

\bibitem{KMPR2014}
Elliot Krop, Sarah Minion, Pritul Patel, and Christopher Raridan.
\newblock A solution to the edge-balanced index set problem for complete odd
  bipartite graphs.
\newblock {\em Bull. Inst. Combin. Appl.}
\newblock Accepted July 17, 2013.

\bibitem{LLT1992}
Sin-Min Lee, Andy~C. Liu, and Sie-Keng Tan.
\newblock On balanced graphs.
\newblock In {\em Proceedings of the {T}wenty-first {M}anitoba {C}onference on
  {N}umerical {M}athematics and {C}omputing ({W}innipeg, {MB}, 1991)},
  volume~87, pages 59--64, 1992.

\end{thebibliography}
\end{document}